\documentclass[11pt,oneside,a4paper]{amsart}

\usepackage[T1]{fontenc} 
\usepackage[utf8]{inputenc} 
\usepackage[english]{babel} 
\usepackage{amsmath} 
\usepackage{mathtools} 
\usepackage{amssymb} 
\usepackage{amsthm} 
\usepackage{csquotes} 
\usepackage{tikz,pgf,pgfplots,mathrsfs} \usetikzlibrary{arrows} \pgfplotsset{compat=1.15} 
\usepackage[a4paper]{geometry} 
\usepackage{graphicx} 
\usepackage{float} 
\usepackage[font=small,justification=centering]{caption} 
\usepackage{mathabx} 
\usepackage{enumitem} 
\usepackage{ifthen} 
\usepackage[nodayofweek]{datetime}
\usepackage[pdftex]{hyperref}
\usepackage{tikz-cd} 

\makeatletter\def\subsection{\@startsection{subsection}{1}\z@{.7\linespacing\@plus\linespacing}
    {.5\linespacing}{\normalfont\scshape\centering}}\makeatother 

\newcounter{claimcount}
\newcounter{makedefcount}
\newcounter{thmcount}

\newtheorem{theorem}{Theorem}[section]
\newtheorem{lemma}[theorem]{Lemma} 
\newtheorem{corollary}[theorem]{Corollary}
\newtheorem{proposition}[theorem]{Proposition}
\theoremstyle{definition}
\newtheorem{definition}[theorem]{Definition}
\newtheorem{remark}[theorem]{Remark}

\newtheorem{question}{Question}

\newcommand*{\numberedtheorem}[2]{\theoremstyle{plain}\newtheorem*{makethm\thethmcount}{#1}
    \begin{makethm\thethmcount}#2\end{makethm\thethmcount}\stepcounter{thmcount}}
\newcommand*{\makedef}[9]{\theoremstyle{definition}\newtheorem*{makedefthm\themakedefcount}{#1}
    \begin{makedefthm\themakedefcount}#2\par#3\par#4\par#5\par#6\par#7\par#8\par#9\end{makedefthm\themakedefcount}\stepcounter{makedefcount}}

\newenvironment{claim*}[1]{\stepcounter{claimcount}\par\noindent\underline{Claim:}\space#1}{}

\newenvironment{shortitem}{\begin{itemize}\vspace{-\parskip}\setlength{\itemsep}{1pt}
    \setlength{\parskip}{0pt}\setlength{\parsep}{0pt}}{\end{itemize}}

\newcommand*{\C}{\mathcal{C}}
\newcommand*{\mfc}{\mathfrak{C}}

\newcommand*{\N}{\mathcal{N}}
\newcommand*{\mfq}{\mathfrak{q}}
\newcommand*{\R}{\mathbf{R}}
\newcommand*{\s}{\mathfrak S}
\newcommand*{\mfs}{\mathfrak{S}}

\newcommand*{\W}{\mathcal{W}}

\newcommand*{\Z}{\mathcal{Z}}

\newcommand*{\lhalf}[1]{\overleftarrow{#1}}
\newcommand*{\rhalf}[1]{\overrightarrow{#1}}

\newcommand*{\nest}{\sqsubset}
\newcommand*{\pnest}{\sqsubsetneq}

\newcommand*{\pconest}{\sqsupsetneq}
\newcommand*{\trans}{\pitchfork}

\DeclareMathOperator{\hthet}{H_\theta}
\DeclareMathOperator{\hull}{hull}
\DeclareMathOperator{\diam}{diam}
\DeclareMathOperator{\dist}{\mathsf{d}}
\DeclareMathOperator{\cay}{Cay}
\DeclareMathOperator{\rel}{Rel}
\DeclareMathOperator{\asdim}{asdim}
\DeclareMathOperator{\stab}{Stab}

\newcommand{\ignore}[2]{\left\{\kern-.7ex\left\{#1\right\}\kern-.7ex\right\}_{#2}}

\tikzset{symbol/.style={draw=none,every to/.append style={edge node={node [sloped, allow upside down, auto=false]{$#1$}}}}}

\begin{document}
\title[Projection complexes and quasimedian maps]{Projection complexes and quasimedian maps}
\author[M Hagen]{Mark Hagen} 
\address{School of Mathematics, University of Bristol, Bristol, UK}
\email{markfhagen@posteo.net}

\author[H Petyt]{Harry Petyt}
\address{School of Mathematics, University of Bristol, Bristol, UK}
\email{h.petyt@bristol.ac.uk}

\date{\today}

\maketitle
\begin{abstract}
We use the projection complex machinery of Bestvina--Bromberg--Fujiwara to study hierarchically hyperbolic groups. In particular, we show that if the group has a BBF colouring and its associated hyperbolic spaces are quasiisometric to trees, then the group is quasiisometric to a finite-dimensional CAT(0) cube complex. We deduce various properties, including the Helly property for hierarchically quasiconvex subsets. 
\end{abstract}

\section{Introduction}
The original motivation for defining hierarchically hyperbolic groups (HHGs) was to build a bridge between the worlds of mapping class groups and cubical groups, providing a framework for studying both. The idea is that these groups can be studied via a ``hierarchy'' of associated hyperbolic spaces: a question about the group is broken into questions about the hyperbolic spaces and the relations between them, and the answers are assembled into an answer to the original question, often using core structure theorems established in \cite{HHS2}. This is a common generalisation of the Masur--Minsky machinery \cite{masurminsky} and of \cite{HHS_I}; it is somewhat related to the work of Kim--Koberda on right-angled Artin groups \cite{kimkoberda}, and also to the work of Bestvina--Bromberg--Fujiwara~\cite{bbf} (which features prominently in the present paper).

When a cubical group admits such a structure (which, in fact, all known cubical groups do \cite{hagensusse}), the associated hyperbolic spaces can always be taken to be quasiisometric to trees (i.e. quasitrees), and this raises the converse question: if the associated hyperbolic spaces are all quasitrees, is the hierarchically hyperbolic group necessarily quasiisometric to a cube complex? We give an affirmative answer under a natural assumption. 

\numberedtheorem{Theorem~\ref{qitocubecomplex}} 
{Any hierarchically hyperbolic group with a \emph{BBF colouring} and whose associated hyperbolic spaces are quasitrees admits a quasimedian quasiisometry to a finite-dimensional CAT(0) cube complex.}

See Section~\ref{maintheorem} for the definition of a BBF colouring and a more precise statement. Proposition~\ref{howtobecolourable} establishes a sufficient condition in terms of separability of stabilisers of product regions for an HHG to virtually admit a BBF colouring.  For the purposes of Theorem~\ref{qitocubecomplex}, a virtual BBF colouring is sufficient, since finite-index subgroups of HHGs are HHGs with the same hierarchically hyperbolic structure \cite{ans}.  Examples of HHGs with virtual BBF colourings (also known as \emph{colourable} HHGs \cite{durhamminskysisto:stable}) include:
\begin{itemize}
    \item hyperbolic groups;
    \item mapping class groups of finite-type oriented hyperbolic surfaces~\cite{HHS2,bbf};
    \item virtually compact special groups (e.g. right-angled Artin and right-angled Coxeter groups)~\cite{HHS_I};
    \item many non-special cocompactly cubulated groups, e.g. the Burger-Mozes~\cite{burgermozes} and Wise~\cite{Wise:csc} lattices in products of trees.
\end{itemize}
In fact, there are few known examples of HHGs that do not virtually admit BBF colourings \cite{hagen:non}.  An interesting question is whether the HHG combination theorems from~\cite{HHS2} and~\cite{berlairobbio} preserve the property of BBF colourability.  The motivating examples of HHGs where the associated hyperbolic spaces are quasi-trees are cubical groups, but there are non-cubical examples.  For instance, in forthcoming work of the first author, Russell, Sisto, and Spriano, it is shown that fundamental groups of non-geometric $3$--dimensional graph manifolds have this property, but not all such groups are cocompactly cubulated~\cite{HagenPrzytycki}.

We actually prove something more general, namely that if we drop the quasitree assumption then we still get a(n equivariant) quasimedian quasiisometric embedding in a finite product of hyperbolic spaces  (Theorem~\ref{qietoproduct}). In the mapping class group case, this, minus the quasimedian property, is a result of Bestvina--Bromberg--Fujiwara \cite[Thm~C]{bbf}, and our work extends their construction. Roughly, our strategy is to show that when the map from \cite{bbf} is composed with projection to one of the factor hyperbolic spaces, then \emph{hierarchy paths} get sent to unparametrised quasigeodesics (Proposition~\ref{hptoqg}). From this, work of Russell--Spriano--Tran \cite{rusptr} lets us deduce that the map is quasimedian (Proposition~\ref{psiqm}). 

Our proof of Theorem~\ref{qietoproduct}, together with~\cite[Thm~4.14]{bbf}, yields \emph{property (QT)}, introduced in the context of mapping class groups and residually finite hyperbolic groups in~\cite{newbbf}:

\begin{corollary}\label{cor:qt}
Let $G$ be a hierarchically hyperbolic group with a BBF colouring and whose associated hyperbolic spaces are quasi-trees.  Then $G$ acts metrically properly by quasimedian isometries on a finite product of quasitrees.
\end{corollary}

Note that Theorem~\ref{qitocubecomplex} does not apply to mapping class groups (except for a few low-complexity examples), as work of Gabai shows that the curve graph is not a quasitree \cite{gabai:almost}. To go from a quasimedian quasiisometric embedding to an actual quasiisometry in the quasitree case, we use a general result, Proposition~\ref{promotingqietoqi}, which we believe to be of independent interest.

It turns out that the quasimedian quasiisometry of Theorem~\ref{qitocubecomplex} can be taken to ``respect convexity'', in the sense that \emph{hierarchically quasiconvex} subsets of the group correspond to convex subcomplexes of the cube complex (Proposition~\ref{hqctoconvex}). Much of the geometric power of CAT(0) cube complexes comes from their convex subcomplexes, and this allows us to transfer that power across to the group. In particular, it allows us to show that such subsets satisfy a coarse version of the Helly property (Theorem~\ref{hellyproperty}), and this leads to bounded packing for subgroups of this type. 

One might hope to establish Theorem~\ref{qitocubecomplex} in greater generality by using the ``cubulation of hulls'' theorem \cite[Thm~2.1]{quasiflats}, which says that the hull of a finite set in a hierarchically hyperbolic space can be approximated by a CAT(0) cube complex. The strategy of proof for that theorem is to take an approximating tree in each associated hyperbolic space and build walls. These approximating trees introduce dependence on the size of the finite set, which would not happen if the hyperbolic spaces were trees to begin with, so one could try to quasicubulate the group by taking uniform cubulations of an exhaustive sequence of finite sets. Unfortunately, the construction of the walls has an essential dependence on the size of the finite set beyond that coming from the individual hyperbolic spaces.

It is natural to wonder to what extent the quasitree condition is necessary: perhaps enough ``quasi-treeness'' is already built into the machinery associated to the BBF colouring. Indeed, in \cite{newbbf} it is shown that mapping class groups act properly on finite products of quasitrees in such a way that orbit maps are quasiisometric embeddings. This raises the following question:

\begin{question}
When is a hierarchically hyperbolic group with a BBF colouring quasi-isometric to a CAT(0) cube complex?
\end{question}

For mapping class groups, it is interesting to ask whether the property of having finitely many conjugacy classes of finite subgroups can be explained using only the hierarchy machinery and the Bestvina-Bromberg-Fujiwara construction, without explicit use of Nielsen realisation as in Bridson's proof \cite{bridson}. That is:

\begin{question}
Does every hierarchically hyperbolic group with a BBF colouring have finitely many conjugacy classes of finite subgroups?
\end{question}

Theorem~\ref{qitocubecomplex} can be viewed as geometric evidence that Question~2 has a positive answer if the associated hyperbolic spaces are quasitrees, and in fact a previous version of this article contained a proof of that special case. A complete proof for all HHGs has since appeared in \cite{haettelhodapetyt:coarse}.

\subsection*{Acknowledgements} 

MH was supported by EPSRC New Investigator Award  EP/R042187/1. HP was supported by an EPSRC DTP at the University of Bristol, studentship 2123260.  We are grateful to the anonymous referee for a number of helpful comments.

\section{Background}

\subsection{Hierarchically hyperbolic groups} 

In the present paper, we deal with \emph{hierarchically hyperbolic groups}, or HHGs. A full definition can be found in \cite[Def.~1.1,~1.21]{HHS2}, and \cite{sistowhatishhs} gives a nice introduction to the concept. Roughly, an HHG consists of a finitely generated group $G$ (with a fixed finite generating set giving a word metric $\dist_G$), a set $\mfs$, and a number $E$ such that the pair $(G,\mfs)$ has the following additional structure.

\begin{shortitem}
\item For each \emph{domain} $U\in\mfs$ there is an associated $E$--hyperbolic space $\C U$ and an $(E,E)$--coarsely Lipschitz $E$--coarsely onto \emph{projection} $\pi_U:G\rightarrow\C U$.
\item $\mfs$ has three relations: \emph{nesting}, $\nest$, a partial order; \emph{orthogonality}, $\bot$, symmetric and antireflexive; and \emph{transversality}, $\trans$, the complement of $\nest$ and $\bot$.
\item $G$ acts cofinitely on $\mfs$, preserving these relations.
\item There is a coarse map $\rho^V_U:\C V\rightarrow\C U$ whenever $U\pnest V$, and a specified subset $\rho^U_V\subset\C V$ of diameter at most $E$ whenever $U\pnest V$ or $U\trans V$.
\item All $g_1,g_2\in G$ have corresponding isometries $g_i:\C U\rightarrow\C g_iU$ such that $g_1g_2=g_1\cdot g_2$, and $g\pi_U(x)=\pi_{gU}(gx)$ and $g\rho^U_V=\rho^{gU}_{gV}$ hold whenever $U\pnest V$ or $U\trans V$ and $x\in G$.
\item For every $x\in G$, the tuple $(\pi_U(x))_\mfs$ is $E$--consistent (defined below).
\end{shortitem}

The axioms also include the existence of a function $E:[0,\infty)\to[0,\infty)$, coming from the \emph{uniqueness axiom} (\cite[Def.~1.1(9)]{HHS2}), whose exact statement is not needed here. When $E$ is referred to as a constant, we are abusing notation and mean $E(0)$. The value of $E(0)$ also depends on the other HHG axioms, such as \emph{partial realisation}; see \cite[Def.~1.1(8)]{HHS2}. When we say that a constant depends only on $E$, we mean that it depends only on the HHG structure. For justifications of the facts that $E(0)$ can be taken to be uniform and that the $\pi_U$ are coarsely onto, see \cite[Rem.~1.6, 1.3]{HHS2}. The above conditions on the isometries $g:\C U\to \C gU$ are simpler than those in the original definition in~\cite{HHS2}, but it is shown in \cite{durhamhagensisto:correction} that they are equivalent, so we just take them as the definition. 

The \emph{complexity} of $(G,\mfs)$ is the (finite) cardinality of a maximal $\pnest$--chain. We now define consistency for tuples; we shall often remark that an inequality holds ``by consistency''.

\begin{definition}[$\kappa$--consistent tuple]
For a number $\kappa\geq E$, let $b=(b_U)\in\prod_\mfs\mathcal P(\C U)$ be a tuple such that every set $b_U$ has diameter at most $\kappa$. We say that $b$ is $\kappa$--consistent if 
\[
\min\big\{\dist_U(b_U,\rho^V_U),\dist_V(b_V,\rho^U_V)\big\}\leq\kappa \hspace{2mm}\text{ whenever } U\trans V, \text{ and}
\]\[
\min\big\{\dist_V(b_V,\rho^U_V),\diam(b_U\cup\rho^V_U(b_V))\big\}\leq\kappa \hspace{2mm}\text{ whenever } U\pnest V.
\]
\end{definition}

Throughout, we write $\dist_U(x,y)$ to mean $\dist_{\C U}(\pi_U(x),\pi_U(y))$ for points (or sets) in $G$. More generally, if $A\subset G$ and $x\in G$, we write $\dist_U(x,A)$ to mean $\dist_{\C U}(\pi_U(x),\pi_U(A))$. If $x\in G$ and $V\trans U$ or $V\pnest U$, then $\dist_U(x,\rho^V_U)$ denotes $\dist_{\C U}(\pi_U(x),\rho^V_U)$.  Let $s\geq100E$. As in \cite[\S2.2]{HHS2}, we say that a domain $U\in\mfs$ is $s$--\emph{relevant} for $x,y\in G$ if $\dist_U(x,y)>s$. We write $\rel_s(x,y)$ for the set of $s$--relevant domains. Any set of pairwise transverse elements of $\rel_s(x,y)$ has a total order $<$, obtained by setting $U<V$ whenever $\dist_U(y,\rho^V_U)\leq E$ \cite[Prop.~2.8]{HHS2}.  

For numbers $r$ and $s$, let $\ignore{r}{s}$ be equal to $r$ if $r\geq s$, and zero otherwise. The \emph{distance formula} for HHGs, \cite[Thm~4.5]{HHS2} states that for any sufficiently large $s$ (in terms of $E$ and the complexity) there are constants $A_s$ and $B_s$ such that, for any $x,y\in G$,
\begin{align}
\frac{1}{A_s}\sum_\mfs\ignore{\dist_U(x,y)}{s}-B_s \leq \dist_G(x,y) \leq A_s\sum_\mfs\ignore{\dist_U(x,y)}{s}+B_s. \tag{DF} \label{hhsdf}
\end{align}

Closely related to the distance formula is the existence of \emph{hierarchy paths}. It is an important theorem that for sufficiently large $D$, any two points in an HHG are connected by a $D$--hierarchy path \cite[Thm~4.4]{HHS2}.

\begin{definition}[Hierarchy path]
In a metric space $X$, a coarse map $\gamma:[0,T]\rightarrow X$ is a $D$--\emph{quasigeodesic}  if it is a $(D,D)$--quasiisometric embedding. A coarse map $\gamma:[0,T]\rightarrow X$ is an \emph{unparametrised} $D$--quasigeodesic if there is a strictly increasing function $f:[0,T]\rightarrow[0,T]$ with $f(0)=0$, $f(T)=T$ such that $\gamma f$ is a $D$--quasigeodesic. A $D$--quasigeodesic $\gamma\subset G$ is a $D$--\emph{hierarchy path} if every $\pi_U\gamma$ is an unparametrised $D$--quasigeodesic.
\end{definition}

It is actually possible to establish the distance formula and the existence of hierarchy paths in the slightly more general setting of \cite[Thms~1.1,~1.2]{bowditchhulls}.

A useful notion in the study of hyperbolic spaces is that of a quasiconvex subset, and there is a natural analogue for HHGs.

\begin{definition}[Hierarchical quasiconvexity]
Let $k:[0,\infty)\rightarrow[0,\infty)$. A set $\Z\subset G$ is $k$--\emph{hierarchically quasiconvex} if every $\pi_U(\Z)$ is $k(0)$--quasiconvex and any point $x\in G$ with $\dist_U(x,\Z)\leq\kappa$ for all domains $U$ is $k(\kappa)$--close to $\Z$. 
\end{definition}

Hierarchically quasiconvex subsets can also be characterised in terms of hierarchy paths: there is a function $k'$ such that $\Z$ is $k$--hierarchically quasiconvex if and only if every $D$--hierarchy path with endpoints in $\Z$ stays $k'(k,D)$--close to $\Z$ \cite[Prop.~5.7]{rusptr}. 

The next two definitions are of prototypical examples of hierarchically quasiconvex subsets. See \cite[Def.~5.15, 6.1]{HHS2} respectively.

\begin{definition}[Standard product region]
For $U\in\mfs$, the \emph{standard product region} associated to $U$ is the nonempty (by the partial realisation axiom of HHGs) set $P_U=\{x\in G:\dist_V(x,\rho^U_V)\leq E \text{ for all } V\trans U\text{, } V\pconest U\}$.
\end{definition}

Note that the equivariance in the definition of an HHG gives $gP_U=P_{gU}$ for all $g\in G$. We also remark that there is a slight formal difference between this definition of product regions and the one in \cite{HHS2}, but this is all the structure we require: despite the name we shall never have cause to refer to the coarse product structure of $P_U$. Recall that in a hyperbolic space $C$, the hull of a subset $A$, denoted $\hull_C(A)$, is the union of all geodesics in $C$ between pairs of points in $A$. 

\begin{definition}[Hull]
For a number $\theta$, the $\theta$--\emph{hull} of a set $A\subset G$, denoted $\hthet(A)$, is the set of points $x\in G$ that project $\theta$--close to $\hull_{\C U}(\pi_U(A))$ in each domain $U$. We give $\hthet(A)$ the subspace metric.
\end{definition}

By \cite[Lem.~6.2]{HHS2}, if $\theta$ is sufficiently large in terms of $E$ then there is a function $k$ such that $\hthet(A)$ is $k$--hierarchically quasiconvex for any set $A$. One can actually use hierarchical quasiconvexity of the hull of a pair of points to prove the distance formula; a proof along these lines can be found in \cite{bowditchhulls}, and the proof in \cite{HHS2} also relies on this property.

\subsection{Median and coarse median spaces} 

\begin{definition}[Median graph, algebra]
A graph is \emph{median} if for any three vertices $x_1,x_2,x_3$ there is a unique vertex $m$, called the median of the triple, such that $\dist(x_i,x_j)=\dist(x_i,m)+\dist(m,x_j)$ whenever $i\neq j$. A \emph{discrete median algebra} is the 0--skeleton of a median graph, equipped with the median operation.
\end{definition}

Though this is not the standard definition of a discrete median algebra, it is an equivalent one, by \cite[Prop.~2.17]{roller}. Note that $\dist(m(x_1,x_2,y),m(x_1,x_2,z))\leq\dist(y,z)$, so the median map is 1--Lipschitz; see \cite[Cor.~2.15]{chatterjidrutuhaglund:kazhdan}, for example.

\begin{remark}\label{mediantocat0}
Every median graph is the 1--skeleton of a unique CAT(0) cube complex \cite[Thm~6.1]{chepoi}, and the 1--skeleton of any CAT(0) cube complex is a median graph. This establishes a three-way correspondence between discrete median algebras, median graphs, and CAT(0) cube complexes.
\end{remark}

We say that a discrete median algebra has \emph{finite rank $R$} if the corresponding CAT(0) cube complex has dimension $R$. Given a subset $A$ of a discrete median algebra $V$, the subalgebra it generates, denoted $\langle A\rangle$, is the smallest subset of $V$ that is closed under the median operation and contains $A$. We say that $A$ is \emph{$M$--median} if the median of any three points in $A$ is $M$--close to $A$. 

For $C>0$, a \emph{$C$--path} in a metric space is a sequence of points $x_0,x_1,\dots,x_n$ such that $\dist(x_i,x_{i+1})\le C$ for all $i$. Following~\cite{bowditchhulls}, we say that a subset $A$ is \emph{$C$--connected} if any two points of $A$ can be joined by a $C$--path.

\begin{proposition} \label{closetohull}
For all numbers $R$, $C$, and $M$, there is a number $H$ such that if $A$ is an $C$--connected $M$--median subset of a discrete median algebra $V$ of rank $R$, then there exists $A'\subset V$ such that $\dist_{Haus}(A,\langle A'\rangle)\leq H$ and $\langle A'\rangle$ is 1--connected.
\end{proposition}

The $C=1$ case of this lemma is proven in \cite[Prop.~4.1]{bowditchhulls}, and there Bowditch remarks that one can show that any $C$--connected $M$--median subset is a bounded Hausdorff distance from the median subalgebra it generates.

\begin{proof}[Proof of Proposition~\ref{closetohull}]
Decompose $A$ as a disjoint union of maximal 1--connected subsets $A_i$. Let $S_i$ be the (nonempty unless $A=A_i$) set of all $A_j$ with $\dist(A_i,A_j)\leq C$, and let $A'$ be the 1--connected set obtained from $A$ by adding, for each $i$, a geodesic of length at most $C$ from $A_i$ to each $A_j\in S_i$. Clearly $\dist_{Haus}(A,A')\leq C$. Since the median map is 1--Lipschitz, $A'$ is $(M+3C)$--median, so by \cite[Prop.~4.1]{bowditchhulls} it is a bounded Hausdorff distance from $\langle A'\rangle$, which is 1--connected by \cite[Lem.~2.1]{bowditchhulls}.
\end{proof}

Although our proof of Proposition~\ref{closetohull} made no explicit reference to the constant $R$, this appears in the proof of \cite[Prop.~4.1]{bowditchhulls}.

A more general notion than median graphs and algebras is that of a \emph{coarse median space}, introduced by Bowditch in \cite{bowditchcoarsemedian}.

\begin{definition}[Coarse median space]
A metric space $(X,\dist)$ is said to be a \emph{coarse median space} if it admits a map $m:X^3\rightarrow X$ with the property that there exists a function $h:\mathbf N\rightarrow[0,\infty)$ such that:
\begin{shortitem}
\item For all $x,y,z,x',y',z'\in X$ we have 
\[
\dist\big((m(x,y,z),m(x',y',z')\big)\leq h(1)\big(\dist(x,x')+\dist(y,y')+\dist(z,z')+1\big).
\]
\item For all $n\in\mathbf N$, if $A\subset X$ has cardinality at most $n$ then there exists a finite discrete median algebra $V$, and maps $f:A\to V$ and $g:V\to X$ such that 
\begin{align*}
&- \hspace{2mm} \dist\big(gm_V(x,y,z),m(g(x),g(y),g(z))\big)\leq h(n) \text{ for all } x,y,z,\in V; \\
&- \hspace{2mm} \dist(a,gf(a))\leq h(n) \text{ for all } a\in A.
\end{align*}
\end{shortitem}
\end{definition}

As mentioned in \cite{bowditchcoarsemedian}, we can also assume that $m$ is symmetric, and that $m(x,x,y)=x$ for all pairs $x,y\in X$. 

Every median graph (i.e. $1$--skeleton of a CAT(0) cube complex~\cite{chepoi}) is a coarse median space~\cite{Bowditch:notes}; the proof of this relies on the fact the convex hull of any finite subset of a CAT(0) cube complex is finite. Hyperbolic spaces are also coarse median spaces, as any finite set of points can be approximated by a tree, and the median function comes naturally from the thin triangles condition. It is shown in \cite[Thm~7.3]{HHS2} that every HHG is a coarse median space, with the function $h$ depending only on $E$. In this case, the median of $x,y,z\in G$ is a point that projects uniformly close to the (coarse) median of the triple $(\pi_U(x),\pi_U(y),\pi_U(z))$ in every hyperbolic space $\C U$.

There is a natural type of structure-preserving map between quasigeodesic coarse median spaces, which Bowditch calls a ``quasimorphism''. We refer to such maps as ``quasimedian'', as in \cite{quasiflats}.

\begin{definition}[Quasimedian map] \label{def:quasimedian}
Let $(X,\dist_X,m_X)$ and $(Y,\dist_Y,m_Y)$ be quasigeodesic coarse median spaces. A map $\psi:Y\rightarrow X$ is $D$--\emph{quasimedian} if
\[
\dist_X\big(\psi m_Y(x,y,z),m_X(\psi(x),\psi(y),\psi(z)\big)\leq D \hspace{2mm} \text{ for all } x,y,z\in Y.
\]
\end{definition}

We conclude this section with a tool for showing that a space is quasiisometric to a CAT(0) cube complex. It says that one can upgrade a quasiisometric embedding in a CAT(0) cube complex to a quasiisometry to a (different) cube complex, as long as the embedding is quasimedian. First we need a simple lemma (which is also observed in~\cite[Section 2]{bowditchhulls}).

\begin{lemma} \label{lem:connectedandmedianimpliesie}
If $Y$ is a 1--connected median subalgebra of a CAT(0) cube complex $X$, then any two points $y,y'\in Y$ can be joined by a $1$--path in $Y$ that is a geodesic of $X$. That is, $Y$ is isometrically embedded.
\end{lemma}

\begin{proof}
We proceed by induction on $\dist(y,y')$. If $\dist(y,y')=1$, then there is nothing to prove, so suppose that $\dist(y,y')=n>1$. Since $Y$ is $1$--connected, there is a $1$--path $P=(p_0,\dots,p_m)$ in $Y$ from $y$ to $y'$. Let $P'=(p'_0,\dots,p'_m)$ be the path in the interval $[y,y']$ of $X$ given by $p'_i=m(y,y',p_i)$. As $Y$ is a median subalgebra, $P'$ lies in $Y$, and $P'$ is a $1$--path because $m$ is $1$--Lipschitz. Thus there exists $j$ such that $\dist(y,p'_j)=n-1$. Since $p'_j\in[y,y']$, we have that $\dist(p'_j,y')=1$. By the inductive hypothesis, $y$ is joined to $p'_j$ by a path in $Y$ that is a geodesic in $X$, and adding the edge from $p'_j$ to $y'$ completes the inductive step.
\end{proof}

\begin{proposition} \label{promotingqietoqi}
If a coarsely connected coarse median space $X$ admits a quasimedian quasiisometric embedding $\Phi$ in a finite-dimensional CAT(0) cube complex $C$, then $X$ is quasiisometric to a finite-dimensional CAT(0) cube complex.
\end{proposition}

\begin{proof}
By finite-dimensionality of $C$, we can perturb $\Phi$ so that its image is contained in $C^0$, the 0--skeleton of $C$, which is a discrete median algebra by Remark~\ref{mediantocat0}. The image $\Phi(X)$ is coarsely connected because $\Phi$ is coarsely Lipschitz and $X$ is coarsely connected. Since $\Phi$ is quasimedian, the median of any three points in $\Phi(X)$ lies uniformly close to $\Phi(X)$. Proposition~\ref{closetohull} shows that $\Phi(X)$ is a bounded Hausdorff distance from a 1--connected median subalgebra $X'$ of $C^0$, which is isometrically embedded by Lemma~\ref{lem:connectedandmedianimpliesie}. Thus $\Phi$ is a quasiisometry from $X$ to the CAT(0) subcomplex of $C$ whose 0--skeleton is $X'$.
\end{proof}

\subsection{Wallspaces, codimension--1 subgroups, and bounded packing} 

A subgroup $H$ of a group $G$ has \emph{bounded packing} if for each distance $R$ there exists a number $N$ such that at most $N$ cosets of $H$ can be pairwise $R$--close in $G$. See \cite{hruskawisebounded} for an account of this.

Haglund and Paulin introduced wallspaces in \cite{haglundpaulin}. The reader is referred to \cite{hruskawisefiniteness} for a survey in line with our usage, as well as for an account of codimension--1 subgroups. 

\begin{definition}[Wallspace]
A \emph{wall} $W$ in a nonempty set $X$ is a partition of $X$ into \emph{halfspaces}, $X=\lhalf W\bigsqcup\rhalf W$. A pair of points in $X$ is \emph{separated} by $W$ if they do not both lie in the same halfspace. A \emph{wallspace} is a set $X$ with a collection $\W$ of walls such that each pair of points is separated by finitely many walls. 
\end{definition}

If $X$ admits an action of a group $G$ then we say a wallspace structure $(X,\W)$ is $G$--invariant if $G\W=\W$. Such a structure leads to an action of $G$ on a CAT(0) cube complex, namely the dual cube complex defined below.

\begin{definition}[Coherent orientation]
An \emph{orientation} of $\W$ is a choice of halfspace for each $W\in\W$. We say the orientation is \emph{coherent} if no two chosen halfspaces are disjoint and every $x\in X$ lies in all but finitely many of the chosen halfspaces. 
\end{definition}

Such orientations are more commonly called \emph{consistent and canonical}, but this conflicts with the terminology for consistency of tuples.

\begin{definition}[Dual complex]
Consider the graph with a vertex for each coherent orientation $p$ of $\W$ and an edge joining $p,q$ whenever they differ on exactly one wall. This graph is median \cite[Prop.~4.5]{nica}, so it is the 1--skeleton of a unique CAT(0) cube complex by Remark~\ref{mediantocat0}. We call this complex the \emph{cube complex dual to} $(X,\W)$.
\end{definition}

\begin{definition}[Codimension--1]
Let $G$ be a finitely generated group with Cayley graph $\cay(G)$. The subgroup $H<G$ is \emph{codimension--1} if for some number $r$ there are at least two $H$--orbits of components $K_i$ of $\cay(G)\smallsetminus \N_r(H)$ with the property that $K_i$ is not contained in the $s$--neighbourhood of $H$ for any $s$. We call such components \emph{deep}.
\end{definition}

Following Sageev \cite{sageev}, one can use the existence of a codimension--1 subgroup to construct a $G$--invariant wallspace structure on $G$, leading to an action of $G$ on the dual cube complex: let $K$ be the $H$--orbit of a deep component in $\cay(G)\smallsetminus \N_r(H)$, and for each translate $gK$ of $K$ by $g\in G$, create a wall by partitioning $G$ as $gK\bigsqcup(G\smallsetminus gK)$.

\section{Cubulation} \label{maintheorem}

In \cite{bbf}, a construction is made that, roughly speaking, takes a collection $\mathbf{Y}$ of geodesic spaces $Y$ with ``projection maps'' between them, and, for any sufficiently large constant $K$, produces a metric space $\C_K\mathbf{Y}$ with an isometrically embedded copy of each $Y$. It is usual to suppress the constant $K$. It turns out that if the $Y$ are uniformly hyperbolic then $\C\mathbf Y$ is hyperbolic \cite[Thm~4.17]{bbf}. Lemma~\ref{projectionaxioms} shows that any set $\mfs_i$ of pairwise transverse domains in an HHG satisfies the conditions of the construction, allowing one to build such a hyperbolic space from them. According to \cite[\S~4.4]{bbf}, a sufficiently nice group action on $\mathbf Y$ induces an action on $\C\mathbf Y$. Lemma~\ref{projectionaxioms} also shows that if $\mfs_i$ is $G$--invariant then the action of $G$ is sufficiently nice, so $G$ acts on $\C\mfs_i$. 

By a \emph{BBF colouring} of an HHG $(G,\mfs)$ we mean a decomposition $\mfs=\bigsqcup_{i=1}^\chi\mfs_i$ into a finite number of $G$--invariant subsets $\mfs_i$ such that any two domains in the same set are transverse. In \cite[\S5]{bbf}, such a colouring is constructed for (a finite-index subgroup of) the mapping class group. It is easy to see that the existence of a BBF colouring is equivalent to the property that no domain is orthogonal to any of its translates. 


The following theorem, which we prove in Section~\ref{sectionproof}, extends \cite[Thm~C]{bbf}. Since their construction only uses the hierarchy structure of the mapping class group, the main addition to their result is the quasimedian property.

\begin{theorem} \label{qietoproduct}
Let $(G,\mfs)$ be a hierarchically hyperbolic group with a BBF colouring $\mfs=\bigsqcup^\chi\mfs_i$. There is a number $K_0=K_0(E)$ such that for any $K>K_0$ there is an equivariant quasimedian quasiisometric embedding $\Psi:G\rightarrow\prod^\chi\C\mfs_i$, with constants depending only on $E$, $K$, and $\chi$. 
\end{theorem}

It follows that the conclusion of Theorem~\ref{qietoproduct} also holds for any group that is virtually an HHG with a BBF colouring, even though such a group need not be an HHG itself \cite{petytspriano:unbounded}. 

Before proving Theorem~\ref{qietoproduct}, we establish sufficient conditions for an HHG to virtually admit a BBF colouring. Note that the equivariance in the definition of an HHG ensures that $\stab_G(U)$ acts on the product region $P_U$ for all $U\in\mfs$.

\begin{proposition} \label{howtobecolourable}
If $(G,\mfs)$ is an HHG such that, for every $U\in\mfs$, the subgroup $\stab_G(U)$ is separable in $G$ and acts cocompactly on $P_U$, then $G$ has a finite-index subgroup that admits a BBF colouring.
\end{proposition}

\begin{remark}
In the case where $G$ is the mapping class group of a finite-type hyperbolic surface, $\mathfrak S$ is the set of isotopy classes of essential (possibly disconnected) subsurfaces~\cite{HHS2}, and a finite-index BBF-colourable subgroup is provided by~\cite[Prop. 5.8]{bbf}. One could also seemingly produce such a subgroup by combining Proposition~\ref{howtobecolourable} with~\cite[Thm 1.5]{leiningermcreynolds}, since the latter result says that multicurve stabilisers in the mapping class group are separable.
\end{remark}

\begin{proof}[Proof of Proposition~\ref{howtobecolourable}]
Let $U_1,\dots,U_n$ be a complete set of representatives of $G$--orbits of domains, so that every product region is $P_{gU_i}=gP_{U_i}$ for some $g\in G$ and some $i$. If $h\in P_{U_i}$ then $1\in P_{h^{-1}U_i}$, so we may assume that $1\in P_{U_i}$. Using cocompactness, fix $t$ such that $P_{U_i}\subset\N_t(\stab_G(U_i))$ for all $i$.

For a constant $R>0$, suppose that $g\not\in\stab_G(U_i)$ but $P_{U_i}$ and $gP_{U_i}$ come $R$--close. Then there exist $a,b\in\stab_G(U_i)$ with $\dist_G(a,gb)<R+10t$, and, in particular, $|a^{-1}gb|<R+10t$. Separability of $\stab_G(U_i)$ gives a finite-index subgroup $G_i$ of $G$ that contains $\stab_G(U_i)$ but no elements with word length less than $20(R+t)$ that represent nontrivial cosets of $\stab_G(U_i)$. Indeed, any of the finitely many such elements is separated from $\stab_G(U_i)$ by a finite-index subgroup, and one can take $G_i$ to be the intersection of these. We have $a^{-1}gb\not\in G_i$, but we also have $a^{-1},b\in\stab_G(U_i)<G_i$. Thus $g\not\in G_i$, and hence $U_i$ and $gU_i$ are not in the same $G_i$--orbit.

The $G_i$ are finite-index subgroups of $G$ such that no two $G_i$--translates of $P_{U_i}$ are $R$--close, so the intersection $G'=\bigcap^nG_i$ is a finite-index subgroup of $G$ such that no two translates of any product region are $R$--close. It suffices to show that, for sufficiently large $R$, no domain is orthogonal to any of its $G'$--translates.

If $U\bot gU$, then let $z\in P_U$. By the partial realisation axiom, there is an element $x\in G$ such that the distances 
\[
\dist_U(x,z), \hspace{5mm} \dist_{gU}(x,gz), \hspace{5mm} \dist_V(x,\rho^U_V), \hspace{5mm} \text{and} \hspace{5mm} \dist_W(x,\rho^{gU}_W)
\]
are all bounded above by $E$ whenever one of $V\pconest U$ or $V\trans U$ and one of $W\pconest gU$ or $W\trans gU$ holds. Then $x$ is close to both $P_U$ and $P_{gU}$, so the distance between $P_U$ and $P_{gU}$ is bounded in terms of $E$. Taking $R$ larger than this bound completes the proof.
\end{proof}

\subsection{Review of the construction from \cite{bbf}} 

Here we summarise the construction from \cite[\S4]{bbf} in the context of Theorem~\ref{qietoproduct}, and state some of the results of that paper that we shall need in what follows. 

Let $\mathbf Y$ be a collection of geodesic metric spaces, with specified subsets $\pi_Y(X)\subset Y$ whenever $X\neq Y$ are elements of $\mathbf Y$. Write $\dist^\pi_Y(X,Z)$ to mean $\diam(\pi_Y(X)\cup\pi_Y(Z))$ for $X\neq Y\neq Z$. We say that $\mathbf Y$ \emph{satisfies the projection axioms with constant $\vartheta$} if 
\begin{flalign}
& \hspace{2mm} \diam(\pi_Y(X))\leq\vartheta \text{ whenever } X\neq Y; \tag{P0} \label{axiom:p0} \\
& \hspace{2mm} \text{ if } X,Y,Z \text{ are distinct and } \dist^\pi_Y(X,Z)>\vartheta, \text{ then } \dist^\pi_Z(X,Y)\leq\vartheta; \tag{P1} \label{axiom:p1} \\
& \hspace{2mm} \text{ for } X\neq Z, \text{ the set } \{Y:\dist^\pi_Y(X,Z)>\vartheta\} \text{ is finite.} \tag{P2} \label{axiom:p2}
\end{flalign}
If, moreover, a group $G$ acts on $\mathbf Y$, and each $g\in G$ induces isometries $g:Y\to gY$, then we say that the projection axioms are satisfied \emph{$G$--equivariantly} if the isometries satisfy $g_1g_2=g_1\cdot g_2$, and if $g\pi_Y(X)=\pi_{gY}(gX)$ holds for any $X,Y\in\mathbf Y$. 

\begin{lemma} \label{projectionaxioms}
Let $(G,\mfs)$ be an HHG with constant $E$, and $\mfs'$ a set of pairwise transverse domains. The set $\{\C U:U\in\mfs'\}$, equipped with $\pi_U(V)=\rho^V_U$ satisfies the projection axioms with constant $s_0+4E$, where $s_0$ is as in \cite[Thm~4.5]{HHS2}. Furthermore, if $\mfs'$ is $G$--invariant, then the projection axioms are satisfied $G$--equivariantly.
\end{lemma}

\begin{proof}
One of the axioms of HHGs is that axiom~\eqref{axiom:p0} holds with constant $E$. For axioms~\eqref{axiom:p1} and~\eqref{axiom:p2}, let $x\in P_U$ and $z\in P_V$ lie in the corresponding product regions. If $\dist^\pi_W(U,V)>4E$, then $\dist_W(x,\rho^V_W)>4E-\diam(\rho^U_W)-\dist_W(\rho^U_W,x)-\diam(\rho^V_W)>E$, so by consistency we have $\dist_V(x,\rho^W_V)\leq E$, and hence $\dist^\pi_V(U,W)\leq4E$, establishing axiom~\eqref{axiom:p1}. Inequality~\eqref{hhsdf} applies with threshold $s_0$. The distance $\dist_G(x,z)$ is finite, so there are only finitely many domains $W$ with $\dist_W(x,z)\geq s_0+2E$. (This can also be deduced from \cite[Lem.~2.5]{HHS2}.) By definition, if $W\neq U,V$, then $\pi_W(x)$ and $\pi_W(z)$ are $E$--close to $\rho^U_W$ and $\rho^V_W$, respectively, so axiom~\eqref{axiom:p2} is satisfied with constant $s_0+2E$. The ``furthermore'' statement is immediate from the equivariance in the definition of an HHG.
\end{proof}

Lemma~\ref{projectionaxioms} applies in the case where $\mfs'=\mfs_i$ comes from a BBF colouring, and we henceforth work in this case. According to \cite{bbf}, there is a constant $\Theta=\Theta(E)$ such that all of the following holds whenever $K\geq \Theta$.

The \emph{quasitree of metric spaces} $\C\mfs_i$ is obtained by taking the disjoint union $\bigsqcup_{\mfs_i}\C U$ and attaching an edge of length $L$ from each point in $\rho^U_V$ to each point in $\rho^V_U$ whenever $\dist^\flat_W(\rho^U_W,\rho^V_W)\leq K$ for all $W\in\mfs_i$, where $\dist^\flat_W$ is a small perturbation of $\dist_W$; more information will be given below. The constant $L$ is determined by $E$ and $K$. By \cite[Theorem~A]{bbf}, each $\C U$ is isometrically embedded and totally geodesic in $\C\mfs_i$. In fact, it is a theorem of Hume \cite{hume} that this quasitree of metric spaces is quasiisometric to a tree-graded space each of whose pieces is quasiisometric to some $\C U$, but we shall not need this for the proof of Theorem~\ref{qietoproduct}.

Importantly for us, since the sets $\s_i$ are all $G$--equivariant, the projection axioms are satisfied $G$--equivariantly. The construction of the quasitree of metric spaces yields a natural action of $G$ on $\C\s_i$. 

A nice property related to this is that $\C\mfs_i$ to some extent preserves the coarse geometry of the component spaces. More specifically, as has been mentioned, since the domains $\C U$ in an HHG are all $E$--hyperbolic, $\C\mfs_i$ is also hyperbolic, with constant depending only on $E$ and $K$ \cite[Thm~4.17]{bbf}. Moreover, if the $\C U$ are all $(\lambda,\lambda)$--quasiisometric to trees, then $\C\mfs_i$ is quasiisometric to a tree, with constant depending only on $K$ and $\lambda$ \cite[Thm~4.14]{bbf}.

For $U,V\in\mfs_i$ and $x\in\C U$, define $\pi^\flat_V(x)$ to be $x$ if $U=V$, and $\rho^U_V$ otherwise. The map $\pi^\flat_V$ coarsely agrees with the closest point projection map to $\C V\subset\C\mfs_i$ \cite[Cor.~4.10]{bbf}.

We now discuss the perturbation of $\dist_U$ mentioned above. For each $U$ there is a ``distance function'' $\dist^\flat_U:\C\mfs_i\times\C\mfs_i\rightarrow[0,\infty)$, which is symmetric, sends the diagonal to zero, and, up to a constant depending only on $E$, satisfies the triangle inequality. By \cite[Thm~3.3]{bbf}, there is a constant $\delta=\delta(E)$ such that for all $x,y\in\C\mfs_i$ we have
\begin{align}
\diam(\pi^\flat_U(x)\cup\pi^\flat_U(y))-\delta\leq \dist^\flat_U(x,y)\leq\diam(\pi^\flat_U(x)\cup\pi^\flat_U(y)). \label{flatdistance}
\end{align}

The quasitree of metric spaces has a distance formula that is somewhat akin to the one for HHGs \cite[Thm~4.13]{bbf}: there is a constant $K'=K'(E,K)>K$ such that for any $x,y\in\C\mfs_i$,
\begin{align}
\frac{1}{2}\sum_{\mfs_i}\ignore{\dist^\flat_U(x,y)}{K'} \leq \dist_{\C\mfs_i}(x,y) \leq 6K+4\sum_{\mfs_i}\ignore{\dist^\flat_U(x,y)}{K}. \label{bbfdf}
\end{align}
A significant difference between inequalities \eqref{hhsdf} and \eqref{bbfdf} is that inequality \eqref{bbfdf} does not have thresholding: one cannot vary the constants $K$ and $K'$ once the space $\C\mfs_i$ has been built.

As a final remark, note that, according to \cite[Thm~4.1]{bbfs}, instead of altering the distance function and using $\dist^\flat$, one could instead perturb the $\rho$--points from the HHG structure and work with the usual distance. This approach gives a tighter distance formula, with threshold $K$ on both sides, but although it simplifies proofs about $\C\mfs_i$ itself, it makes no difference here, so we choose to work with the original HHG structure.

\subsection{The proof of Theorem~\ref{qietoproduct}} \label{sectionproof} 

We work in the notation of the statement of Theorem~\ref{qietoproduct}. That is, $(G,\s)$ is an HHG with a BBF colouring $\s=\bigsqcup_{i=1}^\chi\s_i$. We recall the convention of writing $\dist_U(x,y)$ to mean $\dist_{\C U}(\pi_U(x),\pi_U(y))$.

\begin{lemma} \label{productregions}
There is a domain $U_i\in\mfs_i$ such that $\dist_U(g,\rho^{gU_i}_U)\leq E$ for every $g\in G$ and all $U\in\mfs_i\smallsetminus\{U_i\}$.
\end{lemma}

\begin{proof}
Let $U\in\mfs_i$, let $h$ be any point in the product region $P_U$, and set $U_i=h^{-1}U$, so that $1\in P_{U_i}$. Since $gP_{U_i}=P_{gU_i}$ we are done by the definition of a product region.
\end{proof}

\makedef{Choice of constants}
{   We fix a choice of the constant $K$ with respect to another number, $D$. The reasoning behind this choice will become apparent as the proof progresses. By \cite[Thm~4.4]{HHS2}, there is a constant $D_0$ depending only on $E$ such that any two points in $G$ are joined by a $D_0$--hierarchy path. Let $D>\max\{100E,D_0,\Theta\}$, and choose $K>101D$. We shall fix a particular value of $D$ later, when we have the necessary context.}{}{}{}{}{}{}{}  

\makedef{Construction of $\Psi$}
{   For each $i$, define $\psi_i:G\rightarrow\C\mfs_i$ by setting $\psi_i(g)=g\pi_{U_i}(1)$. Note that $\psi_i(g)\in\C gU_i\subset\C\mfs_i$. Let $\Psi=(\psi_1,\dots,\psi_\chi)$.}{}{}{}{}{}{}{}

The fact that $\Psi$ is a quasiisometric embedding is proved as Proposition~\ref{psiisqi}. This is done in \cite[\S~5.3]{bbf} in the mapping class group case, and we work in the general HHG case. Proposition~\ref{psiqm} shows that $\Psi$ is quasimedian.

\begin{lemma} \label{bbfnearpi}
$\dist_U(\pi^\flat_U\psi_i(g),g)\leq E$ for any $g\in G$ and any $U\in\mfs_i$.
\end{lemma}

\begin{proof}
If $U=gU_i$ then $\pi^\flat_U\psi_i(g)=\pi^\flat_U\pi_{gU_i}(g)=\pi_U(g)$. Otherwise, $\pi^\flat_U\psi_i(g)=\rho^{gU_i}_U$, and we are done by Lemma~\ref{productregions}.
\end{proof}

\begin{lemma} \label{psicoarselipschitz}
$\Psi$ is coarsely Lipschitz, with constants depending only on $E$, $K$, and $\chi$.
\end{lemma}

\begin{proof}
Write $\{h_1,\dots,h_n\}$ for the fixed finite generating set of $G$. Note that $\dist_U(1,h_j)\leq2E$. It suffices to bound $\dist_{\C\mfs_i}(\psi_i(1),\psi_i(h_j))$. By Lemma~\ref{productregions} and inequality \eqref{flatdistance}, if $W\not\in\{U_i,h_jU_i\}$ then we have 
\begin{align*}
\dist^\flat_W(\psi_i(1),\psi_i(h_j)) &\leq\diam\big(\rho^{U_i}_W\cup\rho^{h_jU_i}_W\big) \\ 
&\leq 2E+\dist_W\big(\rho^{U_i}_W,1\big)+\dist_W(1,h_j)+\dist_W\big(h_j,\rho^{h_jU_i}_W\big) \leq 6E.
\end{align*}
Similarly, if $W=U_i\neq h_jU_i$ or $W=h_jU_i\neq U_i$ then $\dist^\flat_W(\psi_i(1),\psi_i(h_j))\leq3E$, and if $W=U_i=h_jU_i$ then $\dist^\flat_W(\psi_i(1),\psi_i(h_j))\leq4E$. It follows from inequality \eqref{bbfdf} that $\dist_{\C\mfs_i}(\psi_i(1),\psi_i(h_j))$ is bounded, as $K>6E$.  
\end{proof}

\begin{proposition}[Quasiisometric embedding] \label{psiisqi}
There is a number $\kappa=\kappa(E,K)$ such that $\Psi$ is a $(\kappa,\kappa)$--quasiisometric embedding.
\end{proposition}

\begin{proof}
By Lemma~\ref{psicoarselipschitz} and the fact that $G$ acts on $\C\mfs_i$ by isometries, it suffices to bound below the distance between $\Psi(1)$ and $\Psi(g)$ by a linear function of $\dist_G(1,g)$. By inequality \eqref{bbfdf}, $2\dist_{\C\mfs_i}(x,y)\geq\sum_{\mfs_i}\ignore{\dist^\flat_U(x,y)}{K'}$ for any $x,y\in\C\mfs_i$. Fix $s>K'+2\delta+4E$ sufficiently large to apply the HHG distance formula, inequality~\eqref{hhsdf}. There are constants $A_s,B_s$ such that $\dist_G(1,g)\leq A_s\sum^\chi Q_i+B_s$, where $Q_i$ denotes $\sum_{\mfs_i}\ignore{\dist_U(1,g)}{s}$.

Consider a domain $U$ contributing to $Q_i$. If $U\not\in\{U_i,gU_i\}$ then by Lemma~\ref{productregions} we have $\dist_U(1,g)\leq2E+\dist_U(\rho^{U_i}_U,\rho^{gU_i}_U)\leq 2E+\diam(\rho^{U_i}_U\cup\rho^{gU_i}_U)$, and hence, by inequality \eqref{flatdistance}, $\dist_U(1,g)\leq \dist^\flat_U(\psi_i(1),\psi_i(g))+2E+\delta$, as any $\rho^V_W$ has diameter at most $E$.

If $U=U_i\neq gU_i$ then $\dist_U(1,g)\leq \dist_U(1,\rho^{gU_i}_U)+E = \dist^\flat_U(\psi_i(1),\psi_i(g))+E$, and similarly if $U=gU_i\neq U_i$. Finally, if $U=U_i=gU_i$, then $\dist_U(1,g)=\dist^\flat_U(\psi_i(1),\psi_i(g))$.

In any case, we have $\dist^\flat_U(\psi_i(1),\psi_i(g)) \geq s-\delta-2E>K'$, so $U$ contributes to the sum in the lower bound of inequality \eqref{bbfdf}. Moreover, $s-\delta-2E>\delta+2E$, so $2\dist^\flat_U(\psi_i(1),\psi_i(g))>\dist^\flat_U(\psi_i(1),\psi_i(g))+\delta+2E>\dist_U(1,g)$. By inequality \eqref{bbfdf}, we thus have 
\[
Q_i \leq 2\sum_{\mfs_i}\ignore{\dist^\flat_U(\psi(1),\psi(g))}{K'} \leq 4\dist_{\C\mfs_i}(\psi_i(1),\psi_i(g)).
\]
Summing over $i\in\{1,\dots,\chi\}$ gives the result.
\end{proof}

Our strategy for showing that $\Psi$ is quasimedian is to show that the $\psi_i$ send hierarchy paths (of a certain quality) in $G$ to unparametrised quasigeodesics in $\C\mfs_i$.

\begin{proposition}[Quasimedian] \label{psiqm}
There is a number $\kappa'=\kappa'(E,K,\chi)$ such that $\Psi$ is $\kappa'$--quasimedian.
\end{proposition}

\begin{proof}
Since the $i^\mathrm{th}$ component of the (coarse) median of a triangle in $\prod^\chi\C\mfs_i$ is the (coarse) median of the $i^\mathrm{th}$ components of the vertices, it suffices to show that every $\psi_i$ is quasimedian. Let $m$ be the median in $G$ of points $x_1,x_2,x_3$. By construction, $\pi_U(m)$ is uniformly close, in terms of $E$, to any geodesic from $\pi_U(x_j)$ to $\pi_U(x_k)$, so there exists $\theta=\theta(E)$ such that $m\in\hthet(x_j,x_k)$. By \cite[Prop.~5.5]{rusptr}, there exists $D_1=D_1(\theta)$ such that $\hthet(x_j,x_k)$ is contained in the union of all $D_1$--hierarchy paths from $x_j$ to $x_k$. Let $D>\max\{D_1,100E,D_0,\Theta\}$, and let $\gamma_{jk}$ be a $D$--hierarchy path from $x_j$ to $x_k$ that contains $m$. 

By Proposition~\ref{hptoqg}, the image of $m$ under $\psi_i$ is contained in each side of the $\mu$--quasigeodesic triangle with vertices $\psi_i(x_j)$ and sides $\psi_i\gamma_{jk}$, in the hyperbolic space $\C\mfs_i$. It follows that $\psi_i(m)$ is uniformly close to the median in $\C\mfs_i$ of $\psi(x_1),\psi_i(x_2),\psi_i(x_3)$. 
\end{proof}

We henceforth fix $D$ as in the proof of Proposition~\ref{psiqm}. Since $D$ is determined by $E$ and $K$, we view dependence on $D$ merely as dependence on $E$ and $K$. Recall that $K>101D$. Perhaps a better way to think of this is that $K$ is chosen sufficiently large in terms of $E$, and then $D$ can take any value between $\max\{D_1,100E,D_0,\Theta\}$ and $\frac{K}{101}$.

\begin{proposition}[Hierarchy paths $\to$ quasigeodesics] \label{hptoqg}
There is a constant $\mu=\mu(E,K)$ such that the image of any $D$--hierarchy path $\gamma:\{0,1,\dots,T\}\rightarrow G$ under $\psi_i$ is an unparametrised $\mu$--quasigeodesic in the hyperbolic space $\C\mathfrak S_i$.
\end{proposition}

The proof of Proposition~\ref{hptoqg} occupies the remainder of this section. Since $G$ is a group and the $\psi_i$ are $G$--equivariant, we may assume that $\gamma(0)=1$, and we write $\gamma(T)=g$. Since $\gamma$ is a $D$--hierarchy path, $\pi_U\gamma$ is a $D$--quasigeodesic, so if $U\in\mfs_i\cap\rel_{100D}(1,g)$ then there exists a minimal $a_U\in\{0,\dots,T\}$ with $\dist_U(\gamma(a_U),1)\geq2D$, and a maximal $b_U\in\{0,\dots,T\}$ with $\dist_U(\gamma(b_U),g)\geq2D$. Moreover, both $\pi_U\gamma|_{[0,a_U]}$ and $\pi_U\gamma|_{[b_U,T]}$ are $10D$--coarsely constant, and since $D>100E$, consistency ensures that $b_V<a_U$ whenever $V<U$. (Recall that $V<U$ if $\dist_V(g,\rho^U_V)\leq E$, in which case $\dist_V(\gamma(b_V),\rho^U_V)\geq2D-E$.) Write $x_U=\gamma(a_U)$, $y_U=\gamma(b_U)$.  By $\gamma_U$ we mean the subpath $\gamma|_{[a_U,b_U]}$.

\begin{lemma} \label{closeonbigguys}
For all $U\in\rel_{100D}(1,g)$, if $x\in\gamma_U$ then $\dist_{\C\mfs_i}(\psi_i(x),\pi_U(x))\leq6K$. In particular, $\psi_i\gamma_U\subset\N_{6K}(\C U)$, and $\psi_i\gamma_U$ is an unparametrised $D'$--quasigeodesic from $\psi_i(x_U)$ to $\psi_i(y_U)$, where $D'=D'(E,K)$.
\end{lemma}

\begin{figure}[h]
\includegraphics[width=10cm]{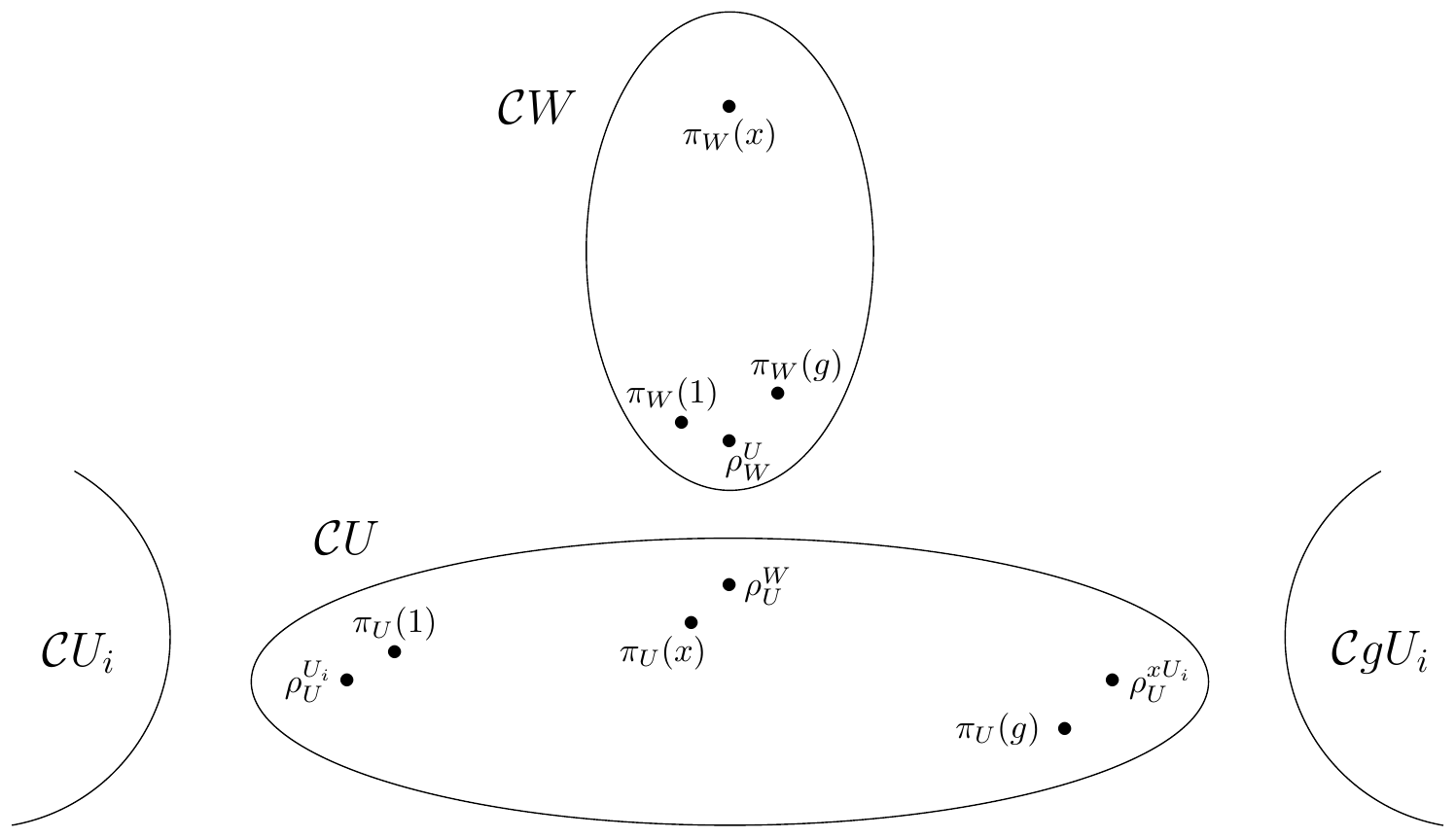}\centering
\caption{Heuristic picture of the proof of Proposition~\ref{closeonbigguys}. The path $\pi_W\gamma$ is a $D$--quasigeodesic.}
\end{figure}

\begin{proof}
By inequality \eqref{bbfdf}, if $\dist_{\C\mfs_i}(\psi_i(x),\pi_U(x))>6K$ then there is some domain $W\in\mfs_i$ such that $\dist^\flat_W(\psi_i(x),\pi_U(x))\geq K$. If $W=U$ then $\dist_W(\pi^\flat_W\psi_i(x),\pi_W(x))\geq K-\delta-2E$ by inequality \eqref{flatdistance}, contradicting Lemma~\ref{bbfnearpi}. We also cannot have $U=xU_i$, for then $\psi_i(x)$ would be equal to $\pi_U(x)$ by definition. If $W\not\in\{U,xU_i\}$ then we have $\diam(\rho^U_W\cup\rho^{xU_i}_W)\geq K$ by inequality \eqref{flatdistance}, so by Lemma~\ref{productregions} we have $\dist_W(\rho^U_W,x)\geq K-3E$. By consistency, we thus have $\dist_U(\rho^W_U,x)\leq E$ in any case where $W\neq U$. 

By definition of $x_U$ and $y_U$, it follows that both $\pi_U(1)$ and $\pi_U(g)$ are $E$--far from $\rho^W_U$, so both $\pi_W(1)$ and $\pi_W(g)$ are $E$--close to $\rho^U_W$. In particular, they are both $(K-4E)$--far from $\pi_W(x)$, which contradicts the fact that $\pi_W\gamma$ is an unparametrised $D$--quasigeodesic.
\end{proof}

For a constant $M>K$, enumerate $\rel_M(1,g)\cap\mfs_i=\{U_2,U_4,\dots,U_n\}$ according to the total order. We abbreviate $a_{U_j}$ to $a_j$ and $\rho^{U_{j-1}}_{U_{j+1}}$ to $\rho^{j-1}_{j+1}$ etc. For odd $j$, let $\alpha_j=\gamma|_{[b_{j-1},a_{j+1}]}$, where we take $b_0=0$ and $a_{n+2}=T$. 

\begin{lemma} \label{alphafellow}
Each path $\psi_i\alpha_j$ is an unparametrised quasigeodesic with constant independent of $M$.
\end{lemma}

\begin{proof}
Let $\hat\alpha_j$ be a geodesic in $\C\mfs_i$ from $\psi_i(y_{j-1})$ to $\psi_i(x_{j+1})$. It suffices to show that $\psi_i\alpha_j$ fellow-travels $\hat\alpha_j$ with constant independent of $M$. 

Our choice of $D$ (large in terms of $E$ and $\Theta$) allows us to invoke \cite[Thm~4.11]{bbf} to get a constant $\mu_1=\mu_1(E,K)$ such that if $U\in\rel_{100D}(1,g)\cap\mfs_i$ and $U_{j-1}<U<U_{j+1}$, then $\hat\alpha_j$ comes $\mu_1$--close to $\pi^\flat_U\psi_i(y_{j-1})$ and $\pi^\flat_U\psi_i(x_{j+1})$. 

By \cite[Cor.~4.10]{bbf}, the map $\pi^\flat_U$ is coarsely Lipschitz, with constants depending only on $E$ and $K$. Since $\pi^\flat_U\pi_{U_{j-1}}(y_{j-1})=\rho^{{j-1}}_U$, this bounds $\dist_{\C\mfs_i}(\pi^\flat_U\psi_i(y_{j-1}),\rho^{{j-1}}_U)$ in terms of $\dist_{\C\mfs_i}(\psi_i(y_{j-1}),\pi_{U_{j-1}}(y_{j-1}))$, which in turn is bounded by Lemma~\ref{closeonbigguys}. But, by Lemma~\ref{closeonbigguys} and the definition of $x_U$, the set $\rho^{j-1}_U$ is uniformly close to $\psi_i(x_U)$ in terms of $E$ and $K$. Thus $\pi^\flat_U\psi_i(y_{j-1})$ lies at uniformly bounded distance from $\psi_i(x_U)$, and similarly $\pi^\flat_U\psi_i(x_{j+1})$ lies at uniform distance from $\psi_i(y_U)$.

This shows that $\hat\alpha_j$ comes uniformly close to both $\psi_i(x_U)$ and $\psi_i(y_U)$. According to Lemma~\ref{closeonbigguys}, it follows that the path $\psi_i\alpha_j|_{[a_U,b_U]}$ fellow-travels a subgeodesic of $\hat\alpha_j$ with constant depending only on $E$ and $K$. By the definition of total ordering on $\rel_{100D}(1,g)\cap\mathfrak S_i$, there is a bound on the overlap between any two of these subpaths. To complete the proof, it suffices to show that their complement is bounded.

Now suppose that $x$ and $y$ are points on $\gamma$ with $\mfs_i\cap\rel_{100D}(x,y)=\varnothing$. Inequality~\eqref{bbfdf} and Lemma~\ref{bbfnearpi} tell us that $\dist_{\C\mfs_i}(\psi_i(x),\psi_i(y))$ is bounded by $6K$, because $K>101D$. Thus each component of the complement of the $[a_U,b_U]$ in $[0,T]$ is mapped by $\psi_i\gamma$ to a set of diameter at most $6K$ inside $\C\mfs_i$.

We have shown that $\psi_i\alpha_j$ lies at Hausdorff distance from $\hat\alpha_j$ bounded in terms of $E,D,K$ only.  Moreover, we have shown that the closest point projection from $\psi_i\alpha_j$ to $\hat\alpha_j$ is coarsely monotone since the above subpaths of $\hat\alpha_j$ have bounded overlaps.
\end{proof}

Our goal is now to apply to $\psi_i\gamma$ a standard lemma for recognising when a piecewise-quasigeodesic in a hyperbolic space is a quasigeodesic:

\begin{lemma}[{\cite[Lem.~4.3]{hagenwise}}]\label{hwlocalgeodesic}
For any numbers $\delta$, $\lambda_1$, $\lambda_2$ and any function $h:\R\to\R$ there is a constant $L_0$ such that the following holds. Suppose that $Z$ is a $\delta$--hyperbolic space and $P=\alpha_1\gamma_2\alpha_3\dots\gamma_{n-1}\alpha_n\subset Z$ is a path such that each $\gamma_i$ is a $\lambda_1$--quasigeodesic and each $\alpha_i$ is a $\lambda_2$--quasigeodesic. Suppose further that for any positive $r$, each of the sets 
\[
\N_{3\delta+r}(\gamma_i)\cap\gamma_{i\pm2} \hspace{5mm} \mathrm{and} \hspace{5mm} \N_{3\delta+r}(\gamma_i)\cap\alpha_{i\pm1}
\]
has diameter at most $h(r)$. If every $\gamma_i$ has length $|\gamma_i|\geq L_0$, then the distance between the endpoints of $P$ is at least $\frac{1}{4\lambda_1\lambda_2}|P|-\lambda_2$. 
\end{lemma}

In the preceding lemma, one can compute $L_0$ explicitly, although we do not need this.  Specifically, if $\lambda$ is such that any $\lambda_i$--quasigeodesic ($i=1,2$) lies $\lambda$--close to the geodesic joining its endpoints, then the required value of $L_0$ depends on $h$ only in the sense that it depends on $h(\delta+3\lambda)$ --- see the proof of~\cite[Lem. 4.3]{hagenwise}.

\begin{lemma} \label{smalloverlaps}
There is a function $h'$, depending on $E$ and $K$, such that for all positive $r$, the diameters of the sets 
\[
\N_r(\psi_i\gamma_j)\cap\N_r(\psi_i\gamma_{j\pm2}) \hspace{5mm} \mathrm{and} \hspace{5mm} \alpha_j\cap\N_r(\psi_i\gamma_{j\pm1}) 
\]
are bounded by $h'(r)$.
\end{lemma}

\begin{proof}
According to \cite[Theorem~A]{bbf}, the subset $\C U$ of $\C\mfs_i$ is totally geodesic, hence quasiconvex.  Thus the usual hyperbolic coarse closest-point projection to $\C U$ is a well-defined coarsely lipschitz coarse map. Furthermore, \cite[Cor.~4.10]{bbf} tells us that if $U\neq V$, then the closest-point projection of $\C U$ onto $\C V$ in $\C\mfs_i$ is in a uniformly bounded neighbourhood of the set $\rho^U_V$. It follows that $\N_r(\C U)\cap\N_r(\C V)$ has diameter bounded in $r$, and by Lemma~\ref{closeonbigguys} this is enough for the first intersection.

By Lemma~\ref{closeonbigguys} and \cite[Cor.~4.10]{bbf}, it follows that the closest point in $\C U_{j+1}$ to $\psi_i(y_{j-1})$ is coarsely equal to $\rho^{j-1}_{j+1}$. Together with the definition of $x_{j+1}$, those two results also show that $\rho^{j-1}_{j+1}$ is coarsely the closest point in $\C U_{j+1}$ to $\psi_i(x_{j+1})$. Thus the closest point projection of $\psi_i\alpha_j$ to $\C U_{j+1}$ is a coarse point, as $\psi_i\alpha_j$ is a quasigeodesic from $\psi_i(y_{j-1})$ to $\psi_i(x_{j+1})$ (Lemma~\ref{alphafellow}), and similarly for $\C U_{j-1}$. The bound for the second intersection follows.
\end{proof}

We are now in a position to prove Proposition~\ref{hptoqg}, which in turn proves Proposition~\ref{psiqm}, and thus completes the proof of Theorem~\ref{qietoproduct}.

\begin{proof}[Proof of Proposition~\ref{hptoqg}]
$\gamma$ is a $D$--quasigeodesic, so by Proposition~\ref{psicoarselipschitz} and the fact that coordinate maps are Lipschitz, the path $\psi_i\gamma$ is a Lipschitz path.

In light of Lemmas~\ref{closeonbigguys}, \ref{alphafellow}, and~\ref{smalloverlaps}, the conditions of Lemma~\ref{hwlocalgeodesic} are met by $\psi_i\gamma=(\psi_i\alpha_1)(\psi_i\gamma_2)\dots(\psi_i\alpha_{n+1})$, with all parameters in terms of $E$ and $K$ only. There thus exists a constant $L_0=L_0(E,K)$ such that if every $\psi_i\gamma_j$ has length at least $L_0$, then $\psi_i\gamma$ is a quasigeodesic, with constants depending only on $E$ and $K$.

Now fix $M=L_0+20K$. The length of $\psi_i\gamma_j$ is at least the distance between its endpoints, $\psi_i(x_j)$ and $\psi_i(y_j)$. By construction and Lemma~\ref{closeonbigguys}, these are, respectively, $10D+6K$--close to $\pi_j(1)$ and $\pi_j(g)$. We have bounded below the length of $\psi_i\gamma_j$ by $\dist_U(1,g)-20D-12K\geq L_0+20K-10D-12K>L_0$, which completes the proof.
\end{proof}

\subsection{The quasitree case} 

Here we work in the setting of Theorem~\ref{qietoproduct}, but with the additional assumption that the domains $\C U$ are uniformly quasiisometric to trees. This case will be our main focus in the remainder of the paper.

\begin{theorem} \label{qitocubecomplex}
Let $(G,\s)$ be an HHG with a BBF colouring $\s=\bigsqcup_{i=1}^\chi\s_i$. If, for all $U\in\s$, the hyperbolic space $\C U$ is $(\lambda,\lambda)$--quasiisometric to a tree, then there is a quasimedian quasiisometry $\Phi$ from $G$ to a finite-dimensional CAT(0) cube complex $\mfc$, with constants depending only on $E$, $K$, $\chi$, and $\lambda$.
\end{theorem}

\begin{proof}
By Theorem~\ref{qietoproduct}, $G$ admits a quasimedian quasiisometric embedding in the product $\prod_{i=1}^\chi\C\s_i$. Since every $\C U$ is a quasitree, \cite[Thm~4.14]{bbf} tells us that the $\C\mfs_i$ are quasitrees. Thus $G$ admits a (nonequivariant) quasimedian quasiisometrc embedding in a finite product of trees, which is a finite-dimensional CAT(0) cube complex. Applying Proposition~\ref{promotingqietoqi} completes the proof.
\end{proof}

To construct $\mfc$ we first applied quasiisometries $\mfq_i$ to the $\C\mfs_i$ to get trees $T\mfs_i$. Let $\pi_i$ be the coordinate map $\mfc\to T\mfs_i$, so that $\phi_i=\pi_i\Phi=\mfq_i\psi_i$. 
Let $\bar\Phi$ be a quasimedian quasiinverse to $\Phi$. Our next proposition strengthens the correspondence between $G$ and $\mfc$ by showing that $\Phi$ ``respects convexity''.

\begin{proposition}[Hierarchically quasiconvex $\leftrightarrow$ convex] \label{hqctoconvex}
For every function $k$ there is a constant $\kappa=\kappa(E,K,\chi,k,\lambda)$ with the property that for each $k$--hierarchically quasiconvex subset $\Z\subset G$ there is a convex subcomplex $\Z'\subset\mfc$ such that $\dist_{Haus}(\Phi(\Z),\Z')\leq\kappa$. Conversely, there is a function $k_0=k_0(E,K,\chi,\lambda)$ such that $\bar\Phi(\Z')$ is $k_0$--hierarchically quasiconvex whenever $\Z'\subset\mfc$ is convex.
\end{proposition}

\begin{proof}
For the first statement, let $\Z'=\{z\in\mfc:\pi_i(z)\in\hull_{T\mfs_i}(\phi_i(\Z))\text{ for all } i\}$. Clearly $\Z'$ is convex and contains $\Phi(\Z)$. It thus suffices to show that every $z\in\Z'$ is uniformly close to $\Phi(\Z)$. Let $z\in\Z'$. Since $\Phi$ is a quasiisometry, there is some $g\in G$ with $\dist_\mfc(\Phi(g),z)$ uniformly bounded and it is enough to bound $\dist_G(g,\Z)$.

By \cite[Prop.~5.7]{rusptr}, $D$--hierarchy paths with their endpoints in $\Z$ stay uniformly close to $\Z$, so their $\phi_i$--images stay uniformly close to $\phi_i(\Z)$. By Proposition~\ref{hptoqg}, their $\psi_i$--images are quasigeodesics, and so their $\mfq_i\psi_i=\phi_i$--images are as well. Hence $\phi_i(\Z)$ is uniformly close to its hull, and it follows from the construction of $\Z'$ that $\dist_{\C\mfs_i}(\psi_i(g),\psi_i(\Z))$ is uniformly bounded. Now, $\pi^\flat_U$ is coarsely Lipschitz, so by Lemma~\ref{bbfnearpi} we have that
\[
\dist_U(g,\Z)\leq\dist_U(g,\pi^\flat_U\psi_i(g)) +\dist_U(\pi^\flat_U\psi_i(g),\pi^\flat_U\psi_i(\Z)) +\dist_U(\pi^\flat_U\pi_i(\Z),\Z)
\]
is uniformly bounded. Finally, hierarchical quasiconvexity of $\Z$ ensures that $g$ is uniformly close to $\Z$.

For the converse statement, note that by \cite[Prop.~5.11]{rusptr}, it suffices to bound the distance between the median $m=m_G(z_1,z_2,g)$ and $\bar\Phi(\Z')$ as the $z_i=\bar\Phi(z'_i)$ vary over $\bar\Phi(\Z')$ and $g$ varies over $G$. By convexity, $m'=m_\mfc(z'_1,z'_2,\Phi(g))\in\Z'$, and so $\bar\Phi(m')\in\bar\Phi(\Z')$. Since $\bar\Phi$ is quasimedian, $\dist_G(\bar\Phi(m'),m_G(z_1,z_2,\bar\Phi\Phi(g))$ is uniformly bounded, and thus so is $\dist_G(\bar\Phi(m'),m)$ as $G$ is a coarse median space.
\end{proof}

\section{Applications} \label{section:applications}

Here we deduce consequences of the construction and results of the previous section. 

\begin{corollary} \label{asdim}
If $(G,\s)$ is an HHG with a BBF colouring $\s=\bigsqcup_{i=1}^\chi\s_i$, then the asymptotic dimension of $G$ is at most $\chi(1+\max\{\asdim(\C U):U\in\mfs)\}<\infty$.
\end{corollary}

\begin{proof}
Each domain has finite asymptotic dimension \cite[Cor.~3.3]{bhsasdim}. Since $G$ acts cofinitely on $\mfs$, there are only finitely many isometry types of domains. We therefore have $\asdim(\C V)\le \sup\{\asdim(\C U):U\in\mfs\} =n<\infty$ for every $V\in\mfs_i$. According to \cite[Thm~4.24]{bbf}, this shows that $\asdim(\C\mfs_i)\le n+1$ for all $i$, and the result follows from subadditivity of asymptotic dimension under direct products \cite[Thm~2.5]{brodskiydydaklevinmitra:hurewicz}.
\end{proof}

The bound of Corollary~\ref{asdim} is not directly comparable with the one from \cite[Thm~5.2]{bhsasdim}, which is based on data such as the length of a maximal $\pnest$--chain in $\mfs$, or the size of certain collections of pairwise $\bot$--related domains---it has no knowledge of the action of $G$ on $\mfs$.

For the remainder, we shall work with an HHG with a BBF colouring and with the property that all its domains are $(\lambda,\lambda)$--quasiisometric to trees. Note that in this case, Corollary~\ref{asdim} also follows from Theorem~\ref{qitocubecomplex} by \cite[Thm~4.10]{wright}. Our next consequence is a coarse version of the Helly property for hierarchically quasiconvex subsets of the group $G$. It relies on the following lemma. Recall that $\hull_\mfc(A)$ denotes the convex hull in $\mfc$ of a subcomplex $A$.

\begin{lemma} \label{hullsofconvexsubcomplexes}
If $\mfc$ is the 1--skeleton of a CAT(0) cube complex of dimension $d$, and $\Z'$ is the 1--skeleton of a convex subcomplex, then $\hull_{\mfc}(\N_r(\Z'))\subset\N_{dr}(\Z')$ holds for all~$r\in\mathbf{N}$.
\end{lemma}

\begin{proof}
We proceed by induction. For the inductive step, note that 
\[
\hull_\mfc(\N_r(\Z')) \subset\hull_\mfc\Big(\N_1\big(\hull(\N_{r-1}(\Z'))\big)\Big) 
        \subset\N_d\big(\N_{d(r-1)}(\Z')\big)
\]
so it suffices to prove the lemma in the case $r=1$.

Let $z\in\hull(\N_1(\Z'))$, so any hyperplane separating $z$ from $\Z'$ meets $\N_1(\Z')$. If two such hyperplanes do not cross then each lies in a halfspace corresponding to the other, so one of them meets $\Z'$, and hence cannot separate $z$ from $\Z'$: a contradiction. Thus any hyperplanes separating $z$ from $\Z$ must pairwise cross, so there are at most $d$ such hyperplanes. In particular, $\dist_\mfc(z,\Z')\leq d$, as $\Z'$ is convex.
\end{proof}

\begin{theorem}[Coarse Helly property] \label{hellyproperty}
Let $(G,\s)$ be an HHG with a BBF colouring and with the property that all its domains are $(\lambda,\lambda)$--quasiisometric to trees. Suppose that $\{\Z_i\}$ is a finite collection of pairwise $R$--close $k$--hierarchically quasiconvex subsets of $G$. Then there is a constant $r=r(E,K,\chi,k,\lambda,R)$ such that for some $g\in G$ we have $\dist_G(g,\Z_i)\leq r$ for all $i$.
\end{theorem}

\begin{proof}
Let $\Phi$ be the quasimedian quasiisometry from $G$ to a finite-dimensional CAT(0) cube complex $\mfc$ provided by Theorem~\ref{qitocubecomplex}. Since the $\Z_i$ are pairwise $R$--close, their $\Phi$--images are pairwise close in terms of $E$, $K$, $\lambda$, and $R$. By Proposition~\ref{hqctoconvex}, each $\Phi(\Z_i)$ is close (in terms of $E$, $K$, $k$, $\lambda$, and $R$) to a convex subcomplex $\Z'_i\subset\mfc$, so the $\Z'_i$ are pairwise close. By Lemma~\ref{hullsofconvexsubcomplexes}, the $\Z'_i$ are uniformly close to a collection of pairwise intersecting convex subcomplexes of $\mfc$, whose total intersection is nonempty by the Helly property for CAT(0) cube complexes \cite[Cor.~3.6]{hruskawisebounded}. One can take $g$ to be the image of any point in this total intersection under any quasiinverse to $\Phi$ that has constants defined in terms of $E$, $K$, and $\lambda$.
\end{proof}

\begin{remark} \label{remark:helly}
The result of Theorem~\ref{hellyproperty} might seem surprising: until recent work of Chepoi-Dragan-Vax\`es~\cite{cdv}, it was not widely known to hold in such a strong way in hyperbolic spaces. Indeed, in that setting, given $n$ quasiconvex subsets that are pairwise close, the na\"ive observation is that there is some point that is $r$--close to each of them, with $r$ depending on $n$.  (The result in~\cite{cdv} removes the dependence on $n$.) 

Even without appealing to~\cite{cdv}, the situation is better for hyperbolic \emph{groups}. For one thing, a strong ``coarse Helly property'' holds for cosets of quasiconvex subgroups, namely bounded packing \cite{gmrs}. Moreover, one can actually remove dependence on the number of quasiconvex subsets as follows.

It is a theorem of Haglund--Wise that for every hyperbolic group $G$ there is a quasiisometry $f$ to a uniformly locally finite CAT(0) cube complex $C$ \cite[Thm~1.8]{HaglundWise:special_combination} (this result uses work of Bonk--Schramm \cite{bonkschramm}). Moreover, $f$ is quasimedian: forgetting the cubical structure of $C$ leaves a hyperbolic space, so, by considering quasigeodesic triangles, $f$ is quasimedian at the level of hyperbolic spaces, and hence is actually quasimedian as any quasigeodesic triangle in $C$ is close to any 1--skeleton geodesic triangle, one of which is a tripod. Now let $\{A_i\}$ be a finite collection of pairwise $R$--close quasiconvex subsets of $G$, so that the $R$--neighbourhoods $\N_R(A_i)$, which are quasiconvex, pairwise intersect. The sets $B_i=f(\N_R(A_i))$ thus pairwise intersect, and, by \cite[Thm~2.28]{haglund}, are uniformly Hausdorff-close to their convex hulls $B_i'\subset C$. By the Helly property for convex subcomplexes of CAT(0) cube complexes, there is some point $p$ contained in every $B_i'$, and hence any point that $f$ maps close to $p$ must lie close to all $A_i$.

Note that Theorem~\ref{hellyproperty} does not provide a new proof of this fact, as the trivial HHG structure on a general hyperbolic group is not an HHG structure whose $\C U$ spaces are quasitrees.
\end{remark}

\begin{corollary} \label{boundedpacking}
Let $(G,\s)$ be an HHG with a BBF colouring and with the property that all its domains are $(\lambda,\lambda)$--quasiisometric to trees. Every hierarchically quasiconvex subgroup $H<G$ has bounded packing.
\end{corollary}

\begin{proof}
By Theorem~\ref{hellyproperty}, any collection of pairwise $R$--close cosets of $H$ are $r$--close to a single point. In particular, these cosets all intersect a ball of radius $r$, which is finite since $G$ is finitely generated. Since cosets are disjoint, we have bounded the number of pairwise $R$--close cosets by the cardinality of the $r$--ball.
\end{proof}

Our final result is analogous to similar statements for hyperbolic groups \cite[Thm.~3.1]{sageev}, \cite[Thm~7.1]{hruskawisefiniteness}, \cite[Thm~5]{nibloreeves}. 

\begin{corollary} 
Let $(G,\s)$ be an HHG with a BBF colouring and with the property that all its domains are $(\lambda,\lambda)$--quasiisometric to trees. If $G$ has a $k$--hierarchically quasiconvex codimension--1 subgroup $H$, then $G$ acts cocompactly on a finite-dimensional CAT(0) cube complex with hyperplane stabilisers commensurable with $H$.
\end{corollary}

\begin{proof}
Fix $R$ large enough that $N=\N_R(H)$ is connected (possible by hierarchical quasiconvexity) and $G\smallsetminus N$ has at least two $H$--orbits of deep components. Let $K$ be one such orbit. This gives $G$ a wallspace structure with walls $gW=\{gK,G\smallsetminus gK)$, so $G$ acts on the CAT(0) cube complex dual to this wallspace, and the hyperplane stabilisers are commensurable with $H$ by \cite[Rem.~3.11, Def.~2.10]{hruskawisefiniteness}. It suffices to show that there are only finitely many orbits of collections of pairwise crossing walls. 

By \cite[Lem.~3.2]{sageev}, if $N$ and $gN$ are disjoint, then the walls $W$ and $gW$ do not cross. Thus if we have a set of pairwise crossing walls $gW$ then the sets $gN$ pairwise intersect, so the cosets $gH$ are pairwise $2R$--close, and we are done by Corollary~\ref{boundedpacking}.
\end{proof}

\bibliographystyle{alpha}
\bibliography{bibtex}

\end{document}